\documentclass[reqno]{amsart}            
\usepackage{amsthm,amssymb,amsfonts,graphicx,cite,color}
\usepackage[bookmarks,bookmarksnumbered,bookmarksopen,colorlinks,linkcolor=blue,citecolor=red]{hyperref}
\usepackage{dsfont}
\newtheorem{thm}{Theorem}[section]

\newtheorem{lem}[thm]{Lemma}

\theoremstyle{definition}

\theoremstyle{remark}
\newtheorem{rem}{Remark}
\numberwithin{equation}{section}

\newcommand{\Real}{\mathbb R}
\newcommand{\eps}{\varepsilon}

\newcommand{\cG}{\mathcal{G}}
\newcommand{\cV}{\mathcal{V}}
\newcommand{\cU}{\mathcal{U}}
\newcommand{\cI}{\mathcal{I}}
\newcommand{\cE}{\mathcal{E}}

\newcommand{\cQ}{\mathcal{Q}}

\renewcommand{\phi}{\varphi}
\newcommand{\ue}{u^\eps}
\newcommand{\uep}{u^{\eps,p}}
\newcommand{\ve}{v^\eps}

\begin{document}
\hyphenation{ap-pro-xi-ma-tion}
\title[Singularly perturbed hyperbolic problems on metric graphs]
{Singularly perturbed hyperbolic problems on metric graphs: asymptotics of solutions}
\author[Yu. Golovaty and V. Flyud]{Yuriy Golovaty${}^1$, Volodymyr Flyud${}^{1,2}$}
\address{\hskip-12pt${}^{1}$  Ivan Franko National University of Lviv, Lviv, Ukraine
\newline
${}^{2}$  Opole University of Technology, Opole, Poland and Ivan Franko National University of Lviv, Lviv, Ukraine}
\email{yu\_holovaty@franko.lviv.ua, flyud@yahoo.com}


\begin{abstract}
We are interested in the evolution phenomena on star-like networks composed of several branches which vary
considerably in physical properties. The initial boundary value problem for singularly perturbed hyperbolic differential equation on a metric graph is studied. The hyperbolic equation becomes degenerate on a part of the graph as a small parameter goes to zero. In addition,
the rates of degeneration may differ in different edges of the graph.
Using the boundary layer method  the complete asymptotic expansions of  solutions are constructed and justified.
\end{abstract}

\maketitle
\section{Introduction}\label{Sect1}
The boundary value problems for ordinary and partial differential ope\-ra\-tors on metric graphs
describe a wide variety of physical processes: vib\-ration and diffusion in networks, wave propagation in  wave\-guide networks,
expansion of signals in neurons etc.
Cur\-rent\-ly, there is increasing interest in  models on graphs, in particular, as a reaction to a great deal of progress in fabricating graph-like structures of the semiconductor materials, (see the survey \cite{Kuchment2004} for details). The idea to investigate the quantum dynamics of particles confined to  metric graphs originated with the study of free electron models of organic molecules \cite{KronigPenney1931, Pauling1936, Platt1949}.
Among the systems those were successfully modeled by graphs we also mention e.g., single-mode acoustic and electro-magnetic waveguide networks \cite{FlesiaJohnstonKunz1987}, the Anderson transition \cite{Anderson1981}, frac\-tion excitations in fractal structures \cite{AvishaiLuck1992}, and mesoscopic quantum systems \cite{KowalSivanEntinWohlmanImry1990}. This resulted into the significant intensification of development of ordinary differential equations as well as PDEs on the metric graphs for the last three decades and numerous publications respectively; we refer the reader to e.g.
\cite{Below1988, Mehmeti1994, MehmetiMeister1989, PokornyiBorovskikh2004, PokornyBook2004}.

It is worth pointing out that the boundary value problems for hyperbolic ope\-rators of the second order on metric graphs as well as for hyperbolic systems of conservation laws on graphs describe a diversity of physical processes.
Such problems arise for example in the modelling of  transversal vibrations of networks,   gas transportation networks,  traffic flow on road networks,  supply chain management,  water flow in open canals etc. (see  \cite{DagerZuazua2002, KostrykinPotthoffSchrader2012,BandaHertyKlar2006}).
The d'Alembert operator $\square=\partial_{t}^2-\Delta$ and the associated Klein-Gordon operator $\square+m^2$ on metric graphs play  an important role in relativistic quantum theories \cite{MehmetiRegnier2003,MehmetiHallerDintelmannRegnier2012}.

Recently, there has also been a growing interest in the singularly perturbed problems on metric graphs. This is partly motivated by  importance of such models for many applied problems in classical and quantum mechanics, theory of non-homogeneous media, scattering theory etc. The differential equations on  graphs with  small or large parameters in their coefficients  represent the natural models of various comp\-li\-cated devices with the irregular ``ramified'' geometry and heterogeneous properties (see e.g. \cite{GolovatyHrabchak2007} and \cite{GolovatyHrabchak2010}).
Also, the asymptotic analysis of Schr\"{o}dinger operators with singular perturbed potentials is  one of the most natural ways to define the Hamiltonians corresponding to  point interactions supported by a discrete set
\cite{GolovatyHryniv2013, Golovaty2013} as well as the Hamiltonians on quantum graphs with point interactions at vertices \cite{Manko2010, Manko2012, ExnerManko2013}.

This paper can be viewed as a natural continuation of our work \cite{FlyudGolovaty2012},
where the vibrations of star-shaped network of strings with  vanishingly small stiffness were treated.
In \cite{FlyudGolovaty2012},  the boundary value problem for  hyperbolic equation containing a small parameter multiplying the second space derivative was studied within the framework of singular perturbation theory, and  asymptotics of solutions were constructed.
The main objective of the present paper is to describe the vibrations of networks  with the essentially different physical properties, e.g., the stiffness coefficients of the strings have a different order of  smallness as a small parameter goes to zero. We study the asymptotic behaviour of solutions to the initial boundary value problem for hyperbolic differential operator on  a star-shaped metric graph with the Dirichlet type boundary conditions.
In contrast to the previous paper, where the case of total degeneration  of the elliptic part was treated, we consider here the partial degeneracy of a hyperbolic operator with different degeneracy factors on subgraphs.
Since the coefficients of the operator depend on a small parameter in the singular way, we show that this leads to a relatively complicated behavior of solutions. We apply the boundary layer methods \cite{VishikLyusternik1957, VasilyevaButuzov1973, Trenogin1970} in order to construct and justify the  asymptotics  of  solutions. Note these results are readily extended to the case of more general  finite graphs.  It is worth to mention that different models of vibrating systems with the singularly perturbed  stiffness  were studied in \cite{Panasenko1987, Sanchez-Palencia1980, GomezLoboNazarovPerez2006}.

The paper is organized as follows. In the next section, we recall the basic notions of metric graphs and PDEs on graphs, and introduce the main object of the paper -- the hyperbolic problem on a star-like graph depending on a small parameter.
In Sections \ref{Sect3} and \ref{Sect4} we construct the formal asymptotic expansion of a solution for the singularly perturbed problem.
Finally, in the last section we justify the asymptotics constructed above.

\section{Statement of problem}\label{Sect2}
A non-oriented finite graph is a pair $\cG=(\cV,\cE)$, where $\cV$ is a finite set of \textit{vertices} and $\cE$ is a finite set of \textit{edges}. We consider a star-shaped planar graph $\cG$ with $n$ edges, i.e., $\cV=\{a,a_1,\ldots,a_n\}$ and $\cE=\{(a,a_1),\ldots,(a,a_n)\}$. The edge $e_j:=(a,a_j)$ can be considered as a piece of line connecting two points $a$ and $a_j$ in $\Real^2$.
We will endow the graph with the metric structure. Any edge $e_j\in \cE$
will be associated with an interval  $[0,\ell_j]$ as follows. Set $\ell_j=\|a_j-a\|$, and suppose that the map
 $\pi_j\colon [0,\ell_j]\to e_j$
is given by
\begin{equation*}
   \pi_j(\tau)=a+\frac{\tau }{\ell_j}\,(a_j-a).
\end{equation*}
The map is a parametrization of $e_j$ by  the arc  length  parameter $\tau$ such that $\pi_j(0)=a$ and $\pi_j(\ell_j)=a_j$.
Hence, there is a canonical distance function $d(x, y)$, $(x, y \in \cG)$ making the graph \textit{a metric space} or \textit{a metric graph.}

We  call  $f:\cG \to\Real$  the \textit{function} on  $\cG$. Here and subsequently,
$f_e$ denotes the restriction of $f$ to the edge $e$ and $f_e(a)$ stands for the limit values
$\lim\limits_{e\ni x\to a}f_e(x)$. The function $f$  is \textit{continuous} on the star-shaped metric graph $\cG$ if it is continuous on each edge $e\in\cE$ and at  the central vertex $a$. The continuity $f$  at the vertex $a$  means that
all limit values $f_e(a)$  along edges $e\in\cE$ equal  the value of $f$  at the vertex.
We also introduce the differentiation of functions along edges.
Set
\begin{equation}\label{EdgeDiff}
  f_{e}'(x)=\frac{d}{d\tau}(f_{e}\circ\pi_e)(\tau),
\end{equation}
where $x=\pi_e(\tau)$.
Let $C(\cG)$ be the space of continuous functions on $\cG$. Let $C^s(e)$ denote the space of functions $f$ such that
$f^{(j)}\colon e \to\Real$ is  continuous on edge $e$ for all $j=0,1,\dots,s$. We also define the spaces
\begin{align*}
  \dot{C}^s(\cG)&=\{f\colon f_e\in C^s(e) \text{ for all } e\in \cE\}, \\
  C^s(\cG)&=\{f\colon f\in C(\cG),\; f_e\in C^s(e) \text{ for all } e\in \cE\}.
\end{align*}
We note that the derivatives of  $f$ at the central vertex  $a$ can not be defined; we must be content with the set of limit values $f_e'(a)$  along the edges. In other words, the first derivative $f'$ of a $C^1(\cG)$-function $f$ is generally discontinuous at $x=a$. Remark that the value $f'(a)$  does not matter in our considerations against the set $\{f'_e(a)\}_{e\in\cE}$.

Let us divide the set $\cE$ of edges into non-empty disjoint subsets $\cE_0,\cE_1,\dots,\cE_k$.  Thereafter
the graph $\cG$ breaks into $k+1$ star subgraphs $\cG_0,\cG_1,\dots,\cG_k$, as is shown in Fig.~\ref{PicGandGi}.
We will denote by $\partial\cG$ the set of vertices $\{a_1,\dots,a_n\}$. Then $\partial\cG=\bigcup_{i=0}^k\partial\cG_i$, where $\partial\cG_i$ is a collection of the vertices  of $\cG_i$ minus $a$. Set also $\cG_*=\bigcup_{i=1}^k \cG_i$. Consequently, $\cV_*=(\partial\cG\setminus\partial\cG_0)\cup\{a\}$ and $\cE_*=\cE\setminus\cE_0$.

\begin{figure}[t]
\centering
  \includegraphics[scale=0.64]{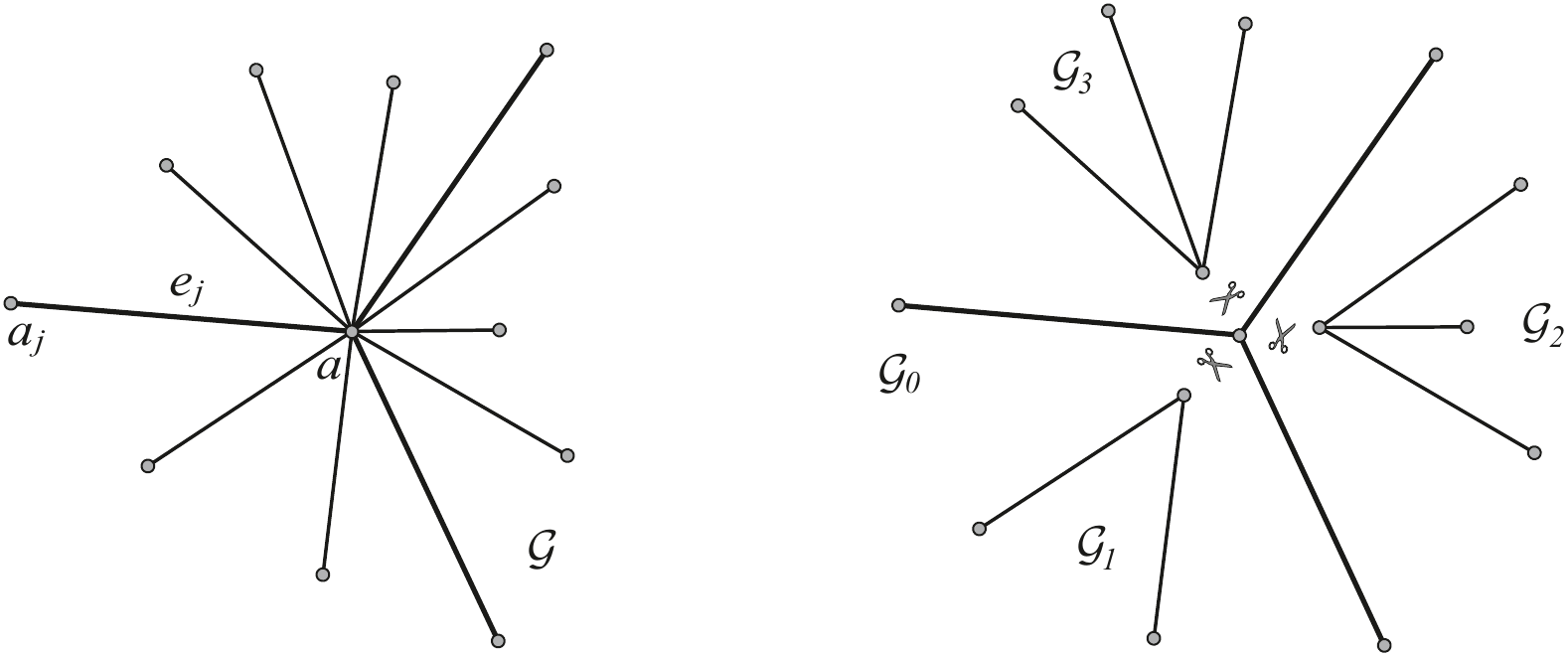}\\
  \caption{The star graph $\cG$ and the subgraphs $\cG_i$.}\label{PicGandGi}
\end{figure}

We consider the function $b^\eps\colon \cG\to\Real$ that is constant on each subgraph $\cG_i$. Set $b^\eps(x)=\eps^{2m_i}$ for $x\in \cG_i$, where $m_0=0$ and $m_1,\dots,m_k$ are positive integers such that $m_1<m_2<\cdots<m_k$, and $\eps$ is a small positive parameter.
The function $b^\eps$ can be considered as the stiffness coefficient of a bundle of strings, which possesses significantly different values on the subsystems $\cG_0,\dots,\cG_k$ as $\eps\to 0$. Remark that the case of rational powers $m_1,\dots,m_k$ can be reduced to the case under consideration by a suitable change of the small parameter $\eps\mapsto \eps^{\alpha}$.
Let $\cI=\Real_t^+$ be the positive time half-line and $\cQ=\cG\times\cI$.

We study the asymptotic behaviour of solution $u^\eps\colon \cQ\to \Real$ of the boundary value problem
\begin{align}
\label{Equation}
    &\partial^{2}_t\ue -\partial_x( b^\eps \partial_x \ue) +q\ue=f\quad \text{ in } \cQ,  \\
\label{InitCond}
    &\ue=\varphi,  \quad \partial_t\ue=\psi \quad \text{ on } \cG\times\{0\},\\
\label{BoundCond}
    &\ue=\mu  \quad \text{ on } \partial\cG\times\cI,\\
\label{Continuity}
    &\ue(\,\cdot\,,t) \text{ is continuous at } x=a \text{ for all } t\in\cI,\\
\label{Kirchhoff}
    & \sum_{e\in\cE} b_{e}^\eps (a)\, \partial_x \ue_e(a,t)=0 \quad \text{  for all } t\in\cI,
\end{align}
where $f\colon \cQ\to\Real$, $\mu\colon \partial\cG\times\cI\to \Real$ and $q, \phi, \psi\colon \cG\to\Real$
are given functions. By $\partial_x^j $ we mean the operator of differentiation along an edge as in \eqref{EdgeDiff}. Since $b^\eps$ is constant on each subgraph $\cE_i$, \eqref{Equation} is actually the set of differential equations
\begin{equation*}
  \partial^{2}_t\ue_e -\eps^{2m_i}\partial_x^2 \ue_e +q_e(x)\ue_e=f_e(x,t)\quad \text{ in } e\times\cI,
\end{equation*}
where $e\in\cE_i$ and $i\in\{0,\dots,k\}$. Condition \eqref{Kirchhoff}  is usually called the~\textit{Kirchhoff vertex condition.}

In order to obtain both a smooth enough solution $\ue$ of \eqref{Equation}--\eqref{Kirchhoff}  for each positive $\eps$ and the complete asymptotic expansion of $\ue$, it is necessary to put some restrictions on the input data. We suppose that
\begin{equation}\label{AssumptionFQ}
q \in \dot{C}^\infty(\cG), \quad f\in \dot{C}^\infty(\cQ), \quad \phi, \,\psi\in C^\infty(\cG), \quad \mu\in C^2(\partial\cG\times\cI).
\end{equation}
Moreover, the following compatibility conditions
\begin{equation}\label{Fitting01}
\begin{aligned}
  &\phi(a_j)=\mu(a_j,0), \quad \psi(a_j)=\partial_t\mu(a_j,0) \quad\text{ for } a_j\in\partial\cG,\\
  &\sum_{e\in\cE_i} \phi_e'(a)=0  \quad\text{ for } i\in\{0,\dots,k\}
\end{aligned}
\end{equation}
hold.
\begin{rem}
It may be mentioned here that we can make assumptions upon the input data to ensure
the existence of a $C^2(\cQ)$ solution $\ue$ for all $\eps\in(0,1)$. Namely, \eqref{Fitting01}
must be  supplemented with the following conditions
\begin{equation}\label{Fitting2}
\begin{aligned}
   &\partial^2_t \mu(a_j,0)-\phi''(a_j)+q(a_j)\phi(a_j)=f(a_j,0) && \text{for } a_j\in\partial\cG_0,\\
  &\partial^2_t \mu(a_j,0)+q(a_j)\phi(a_j)=f(a_j,0),\quad \phi''(a_j)=0  && \text{for } a_j\in\partial\cG\setminus\partial\cG_0,  \\
  &\sum_{e\in\cE_i} \psi_e'(a)=0 \quad \text{ for } i\in\{0,\dots,k\}, \qquad \phi_{e}''(a)=0  && \text{for } e\in\cE,
\end{aligned}
\end{equation}
and the function $q(x)\phi(x)-f(x,0)$ must be continuous at $x=a$.
But these conditions are relatively restrictive. Therefore we assume below that only $C^1$ compatibility conditions \eqref{Fitting01} hold.
\end{rem}

It is well-known \cite[Ch.7.2]{EvansPDE} that the singularities of solutions of hyperbolic equations propagate along cha\-rac\-teristics. Since the initial data of \eqref{Equation}--\eqref{Kirchhoff} are smooth enough, the single reason for discontinuity of  $\ue$ and its derivatives is a mismatch
of the initial conditions and boundary value ones at vertices of the graph.
Suppose that the $C^1$ compatibility conditions \eqref{Fitting01} are satisfied. If conditions \eqref{Fitting2} do not hold, but there exist finite values of all quantities $\partial^2_t \mu(a_j,0)$, $\phi''(a_j)$, $f(a_j,0)$, \dots in these equalities, then the second derivatives of solution can possess only the jump discontinuities along characteristics starting at the vertices. Therefore the second derivatives exist almost everywhere in $\cQ$ and remain bounded on each compact subset of $\cQ$.
We introduce the space
\begin{equation*}
\cU(\Omega)=\left\{u\in C^1(\Omega)\colon \partial^2_t u,\; \partial_{x}\partial_{t}u,\; \partial^2_x u\in L^\infty_{loc}(\Omega)\right\}.
\end{equation*}
From now on, by a solution of the hyperbolic equation we mean a \textit{weak solution} that belongs to $\cU(\cQ)$.

Conditions \eqref{Fitting01} together with continuity of  the initial data $\phi$ and $\psi$ at $x=a$ ensure existence of a unique weak solution $\ue$ of {\eqref{Equation}--\eqref{Kirchhoff}} \cite{MehmetiMeister1989,Mehmeti1994,KostrykinPotthoffSchrader2012}. Remark that all well-known results \cite{KrzyzanskiSchauder1936,SaramotoI1970,SaramotoII1970,LadyzhenskayaBVPs,VolevichGindikin1996} on the unique solvability for hyperbolic initial boundary value problems are still true for the problems on metric graphs.

\section{Formal asymptotic expansion of solution: leading terms}\label{Sect3}
\subsection{Limit problem}
We look for an approximation to the solution $\ue$ of  \eqref{Equation}--\eqref{Kirchhoff} for $\eps$ small enough, and  begin by setting
\begin{equation*}
  \ue(x,t)= u(x,t)+o(1) \quad \text{ as } \eps\to 0.
\end{equation*}
Recall that the function $b^\eps$ vanishes on the subgraph $\cG_*$ as $\eps$ goes to zero.
Upon substituting $\eps=0$ into \eqref{Equation} we see that $u$ must satisfy the following equations
\begin{equation*}
    \partial^2_t u-\partial^2_x u+qu= f \quad \text{ in }  \cQ_0, \qquad
    \partial^2_t u+qu= f \quad\text{ in } \cQ_*,
\end{equation*}
where  $\cQ_0= \cG_0\times\cI$ and $\cQ_*=\cG_*\times \cI$.
In addition, our sending $\eps\to 0$ in \eqref{Kirchhoff} yields the condition
\begin{equation*}
   \sum_{e\in\cE_0} \partial_x u_{e}=0 \quad \text{ on } \{a\}\times\cI.
\end{equation*}
Thus on the subgraph $\cG_*$  hyperbolic equation \eqref{Equation} degenerates  into the ordinary differential equation  with respect to  $t$, depending on parameter $x$. Therefore  $u$  can not satisfy all of the  boundary conditions \eqref{BoundCond}--\eqref{Kirchhoff}.
It is reasonable to consider the problem
\begin{align}\label{ProblemU0}
   & \partial^2_t u-\partial^2_x u+qu= f \qquad \text{in }  \cQ_0,\\
   &\partial^2_t u+qu= f \qquad \text{in } \cQ_*,\\
   & u=\varphi, \quad \partial_tu=\psi \quad \text{on } \cG\times\{0\},\qquad
    u=\mu\quad \text{ on } \partial\cG_0\times\cI,\\
   & \text{the restriction of } u(\,\cdot\,,t) \text{ to } \cG_0
   \text{ is continuous at } a \text{  for all } t\in\cI,\\\label{ProblemU0last}
   &  \sum\limits_{e\in\cE_0} \partial_x u_{e}=0 \quad \text{on } \{a\}\times\cI,
\end{align}
which is actually an uncoupled system of the initial boundary value problem for the hyperbolic equation on the graph $\cG_0$ and ordinary differential equations on edges of $\cG_*$.

\begin{lem}\label{LemmaU}
Under assumption  \eqref{AssumptionFQ}, \eqref{Fitting01} there exists a unique weak solution of \eqref{ProblemU0}--\eqref{ProblemU0last}.
\end{lem}
\begin{proof}
  First, the restriction $u$ to $\cQ_0$ solves the hyperbolic boundary value problem on $\cG_0$:
\begin{align}\label{ProblenU0G1}
  &\partial^2_t u-\partial^2_x u+qu= f \qquad \text{in }  \cQ_0,\\\label{ProblenU0G1init}
  &u=\varphi,  \quad \partial_tu=\psi \; \text{ on } \cG_0\times\{0\},\quad
  u=\mu \; \text{ on } \partial\cG_0\times\cI,\\
  &u(\,\cdot\,,t) \text{  is continuous at } x=a \text{  for all } t\in\cI,\\\label{ProblenU0G1last}
  & \sum\limits_{e\in\cE_0} \partial_x u_{e}=0 \quad \text{on } \{a\}\times\cI.
\end{align}
The problem admits a unique weak solution, because the input data are smooth enough and the $C^1$ compatibility conditions \cite{KostrykinPotthoffSchrader2012}
\begin{equation}\label{FittingU0G1}
\phi(a_j)=\mu(a_j,0), \quad \psi(a_j)=\partial_t\mu(a_j,0) \text{ \ for } a_j\in\partial\cG_0,\quad
\sum_{e\in\cE_0} \phi_e'(a)=0
\end{equation}
hold, which follows from \eqref{AssumptionFQ} and \eqref{Fitting01}.

Next, we can find $u$ on the rest of edges by solving the Cauchy problems for ordinary differential equations with respect to time
\begin{equation}\label{ProblemU0G2}
\partial^2_t u+qu= f \quad \text{ in } e\times\cI,\qquad
 u=\varphi,  \quad \partial_tu=\psi \quad \text{ on } e\times\{0\}
\end{equation}
for each $e\in \cE_*$ separately. All these problems can be solved explicitly:
\begin{align}\label{U0QnonZero}
 &\begin{aligned}
  u(x,t)=\phi(x&)\cos\sqrt{q(x)}\,t+\frac{\psi(x)}{\sqrt{q(x)}}\sin\sqrt{q(x)}\,t\\
  &+ \frac{1}{\sqrt{q(x)}} \int_0^tf(x,\tau)\sin\sqrt{q(x)}(t-\tau)\,d\tau, \quad \text{ if }  q(x)\neq 0,
 \end{aligned}\\\label{U0QZero}
 &u(x,t)=\phi(x)+t\psi(x)+\int_0^t(t-\tau)f(x,\tau)\,d\tau,\quad \text{ if }  q(x)=0
\end{align}
for all $x\in \cG_*$ and $t\in\cI$.
Here we choose the single-valued branch of the root $w=\sqrt{z}$ such that $w(1)=1$.
It is easy to check that the solution $u$ given by \eqref{U0QnonZero} and \eqref{U0QZero} is a real-valued $C^\infty$-function on each edge
$e\in \cE_*$, since the functions $q$, $f$, $\phi$ and $\psi$ are smooth  by \eqref{AssumptionFQ}. Therefore the restriction $u$ to the subgraph $\cQ_*$ belongs to $\dot{C}^\infty(\cQ_*)$. Returning now to the whole graph $\cG$, we see that $u$ is a solution of  \eqref{ProblemU0}--\eqref{ProblemU0last}. \end{proof}

From now on, $U_0$ stands for the solution of \eqref{ProblenU0G1}--\eqref{ProblenU0G1last} and $u_0$ for the solution of \eqref{ProblemU0G2}.
Thus
\begin{equation}\label{RegularAsymptoticU0}
  u(x,t)=
  \begin{cases}
    U_0(x,t)& \text{if } (x,t)\in \cQ_0,\\
    u_0(x,t)& \text{if } (x,t)\in \cQ_*.
  \end{cases}
\end{equation}
Because we can not control in time the value of $u_0$ both at the central vertex $a$ and on the boundary of graph $\cG_*$, the approximation \eqref{RegularAsymptoticU0} generally ignores continuity condition \eqref{Continuity}  as well as boundary conditions \eqref{BoundCond} on $\partial\cG_*$.  This approximation is therefore not suitable for the whole  graph $\cG$. We shall refine \eqref{RegularAsymptoticU0} by applying  boundary layer approximations.

\begin{figure}[t]
  \includegraphics[scale=.64]{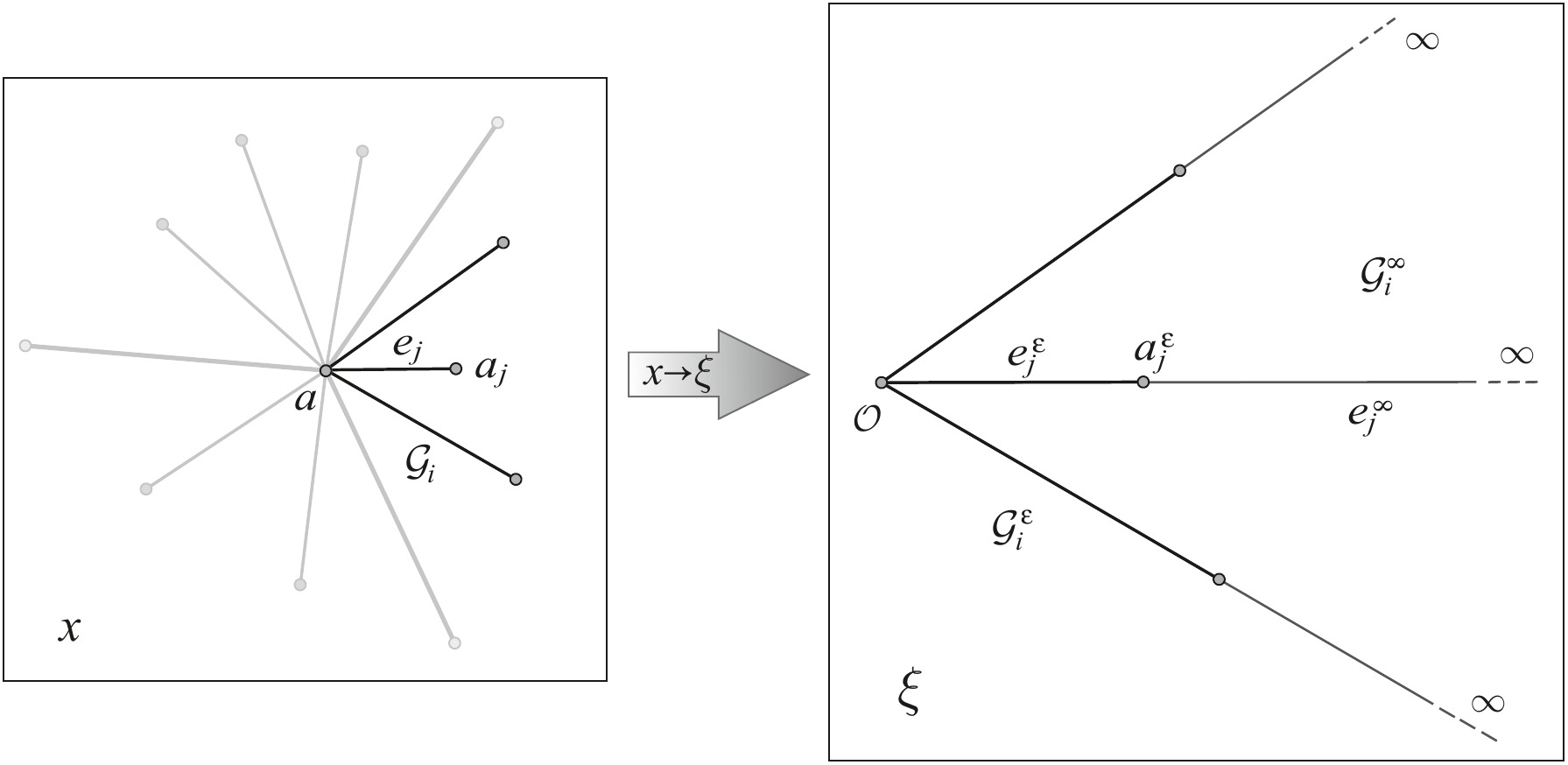}\\
  \caption{The graphs $\cG_i^\eps$ and $\cG_i^\infty$.}\label{PicGandGinfty}
\end{figure}

\subsection{Boundary layers in a vicinity of the central vertex}
First we will modify the approximation \eqref{RegularAsymptoticU0} in the area of degeneration of the hyperbolic equation in order to satisfy the  continuity condition \eqref{Continuity}.

Given a subgraph $\cG_i$, we introduce the new variables $\xi=\eps^{-m_i}(x-a)$. In the auxiliary space $\Real^2$ with the variables $\xi$  we consider
the star graph $\cG_i^\eps$, which is the image of $\cG_i$ under the expanding mapping $x\mapsto \xi$.
For all positive $\eps$ this image lies on the noncompact star graph $\cG_i^\infty=(\{O\}, \cE^\infty_i)$, as shown in Fig.~\ref{PicGandGinfty}.
Therefore the star-like graph $\cG_i^\infty$ has the same number of edges as $\cG_i$, but its edges are  noncompact, i.e., the edges are rays with the origin at point $O$ of the auxiliary plane $\Real^2$. Let $\cE^\infty_i$ be  the set of edge-rays $e_j^\infty$ of $\cG_i^\infty$.
All such  edges $e_j^\infty$ possess the natural parametrisation
\begin{equation*}
\pi_j^\infty\colon [0,+\infty)\to e^\infty_j,\qquad  \pi_j^\infty(\tau)=\frac{a_j-a}{\|a_j-a\|}\,\tau.
\end{equation*}
Since $\partial_{\xi}=\eps^{m_i}\partial_{x}$, in terms of the new variables the homogeneous PDE \eqref{Equation} on $\cG_i$ can be written as
\begin{equation}\label{PDEinFastVariables}
  \partial^2_t\ve  -\partial^2_{\xi}\ve+ q(a +\eps^{m_i}\xi)\ve=0\quad \text{ in } \cG_i^\eps\times\cI
\end{equation}
for all $i=1,\dots,k$.

The coefficient $b^\eps$ in \eqref{Equation} is infinitely small on  $\cG_*$ as $\eps\to 0$, therefore we can use $b^\eps$ as a ``small parameter'' in this subgraph.
Let us introduce the  fast variables $y_\eps=(x-a)/\sqrt{b^\eps(x)}$ on the whole graph $\cG_*$, which coincide with $\xi=\eps^{-m_i}(x-a)$ on each subgraph $\cG_i$,  and modify our approximation
\begin{equation}\label{ApproxU0V0}
  \ue(x,t)\approx
  \begin{cases}\displaystyle
    U_0(x,t)& \text{if } (x,t)\in \cQ_0,\\
    u_0(x,t)+v_0(y_\eps,t)& \text{if } (x,t)\in \cQ_*.
  \end{cases}
\end{equation}

We define the function $v_0$ as follows.
Fix a number $i\in\{1,2,\dots,k\}$ and an edge $e\in \cE_i$, and let  $e^\infty$ be the corresponding edge-ray of $\cE^\infty_i$ in $\Real^2$ with the coordinates $\xi=\eps^{-m_i}(x-a)$ (see Fig.~\ref{PicGandGinfty}).
In view of  \eqref{PDEinFastVariables} and  the formal Taylor expansion
\begin{equation}\label{Qexpansion}
  q_{e}(a +\eps^{m_i}\xi)=q_{e}(a)+\eps^{m_i}q'_{e}(a)\xi+\textstyle\frac{1}{2}\eps^{2m_i}q''_{e}(a)\xi^2 +\cdots,
\end{equation}
we assume that the restriction of $v_0$ to the edge $e^\infty$ solves the initial boundary value problem
\begin{equation}\label{ProblemV0}
\begin{aligned}
&\partial^2_t v -\partial^2_{\xi}v+ q_{e}(a)v=0 &&\text{in } e^\infty\times \cI,\\
&v=0, \;\; \partial_t v=0  && \text{on } e^\infty\times \{0\}, \\
&v=U_0(a,\cdot)-u_{0e}(a,\cdot)  && \text{on } \{O\}\times \cI.
\end{aligned}
\end{equation}
The value $U_0(a,t)$ is  uniquely defined, since $U_0$ is continuous at the vertex $a$.
Hence, the approximation \eqref{ApproxU0V0} satisfies \eqref{Continuity}, since
$u_{0e}(a,t)+v_{0e}(O,t)=U_0(a,t)$ for all $e\in\cE_*$ and $t\in\cI$ .
Note that the initial data of \eqref{ProblemV0} satisfy the compatibility conditions:
\begin{align*}
  &\lim_{t\to 0}v(O,t)=U_0(a,0)-u_{0e}(a,0)=\phi(a)-\phi(a)=0,\\
   &\lim_{t\to 0}\partial_t v(O,t)=\partial_tU_0(a,0)-\partial_tu_{0e}(a,0)=\psi(a)-\psi(a)=0,
\end{align*}
 which  follows from \eqref{ProblenU0G1init}, \eqref{ProblemU0G2} and continuity of $\phi$ and $\psi$.

\begin{lem}\label{LemmaV0isBoundaryLayer}
Let $v_0$ be a solution of  \eqref{ProblemV0}.
 Given $T>0$, the composition
 $$
    V_\eps(x,t)=v_0\left(\frac{x-a}{\sqrt{b^\eps(x)}},t\right)
$$
represents a boundary layer function near the central vertex as $\eps\to 0$, i.e., $V_\eps$ is different from zero
in the $\sqrt{b^\eps}\, t$-neighbourhood of  the vertex $a$ only.
\end{lem}

\begin{proof}
 Problem \eqref{ProblemV0}  is a standard  initial boundary value problem for hyperbolic equation with constant coefficients in a quarter-plane \cite[II.2]{TikhonovSamarskiiHandbook} of the form
\begin{equation}\label{ProblemInQuarter}
    \begin{aligned}
      &\partial^2_t v -\partial^2_{s}v-\vartheta v=0 &&\text{in } \{(s,t)\in \Real^2\colon s>0,\,t>0\},\\
      &v(s,0)=\alpha(s),\quad \partial_t v(s,0)=\beta(s)  && \text{for } s>0, \\
      &v(0,t)=\nu(t)  && \text{for } t>0.
    \end{aligned}
\end{equation}
It admits a unique solution, provided the $C^1$ compatibility conditions
\begin{equation}\label{CompatibilityCondForV}
\nu(0)=\alpha(0),\qquad \nu'(0)=\beta(0)
\end{equation}
hold.
Under the characteristic $s-t=0$ passing through the origin, where the boundary value condition $\nu$ have no effect on the solution, we have
\begin{equation}\label{IntegralRepresentForSoln}
  v(s,t)=\frac{1}{2}\left(\alpha(s+t)+\alpha(s-t)\right)
  +\frac{1}{2}\int\limits_{s-t}^{s+t}\big(k(t,s,y)\beta(y)+\partial_t k(t,s,y)\alpha(y)\big)\,dy,
\end{equation}
where $k(t,s,y)=J_0\left(\sqrt{\vartheta\big((s-y)^2-t^2\big)}\right)$ and
$J_0$ is the Bessel function (see \cite[II.5]{TikhonovSamarskiiHandbook}). The last formula is valid for all real $\vartheta$, in particular, for $\vartheta=0$  it turns into d'Alembert's formula.

Hence, for $s-t>0$ the solution $v(s,t)$  depends on the initial data $\alpha$ and $\beta$ in the interval $[s- t, s+t]\subset\Real_+$ only.
From this we deduce that a solution $v_0$ of \eqref{ProblemV0} is equal to zero for $s>t$, because it satisfies the homogeneous initial conditions. In other words, $v_0$ can be different from zero in the $t$-neighbourhood of the origin $\mathcal{O}$ only. Then $V_\eps$ is  nonzero on the set $\{(x,t)\colon t>0,\, x-a<\sqrt{b^\eps(x)}\,t\}$ only, i.e., in the $\sqrt{b^\eps}\, t$-neighbourhood of  the vertex $a$ only.
This region is small as $\eps\to 0$ uniformly on $t\in (0,T)$.
\end{proof}

\subsection{Improvement of the asymptotics near the boundary $\partial\cG_*$}

We now modify  asymptotics \eqref{ApproxU0V0} near the boundary vertices $a_j\in\partial\cG_*$,
in the same way as we did it near the central vertex $a$.
Let $e$  be the edge of subgraph $\cG_i$ connecting  the vertices $a$ and $a_j$.
We introduce the new variables
\begin{equation*}
  z=\eps^{-m_i}(x-a_j) \text{ \ for } x\in e.
\end{equation*}
Suppose that the function $w_{0,e}$ solves the initial boundary value problem in the quarter-plane
\begin{equation}\label{ProblemW0}
\begin{aligned}
&\partial^2_t w -\partial^2_{z}w+ q_{e}(a_j)w=0 &&\text{in } \{(z,t)\colon z<0, \,\,t>0\},\\
&w(z,0)=0,\quad \partial_t w(z,0)=0 &&\text{for } z<0, \\
&w(0,t)=\mu_e(t) -u_{0e}(a_j,t) &&\text{for } t>0.
\end{aligned}
\end{equation}
In view of \eqref{Fitting01} and \eqref{ProblemU0G2}, the $C^1$ compatibility conditions for the  problem hold. As in the earlier proof of Lemma~\ref{LemmaV0isBoundaryLayer}, one likewise deduces that $w_{0,e}$ is nonzero above the characteristic $z+t=0$ only.

\begin{figure}[t]
  \centering
  \includegraphics[scale=0.64]{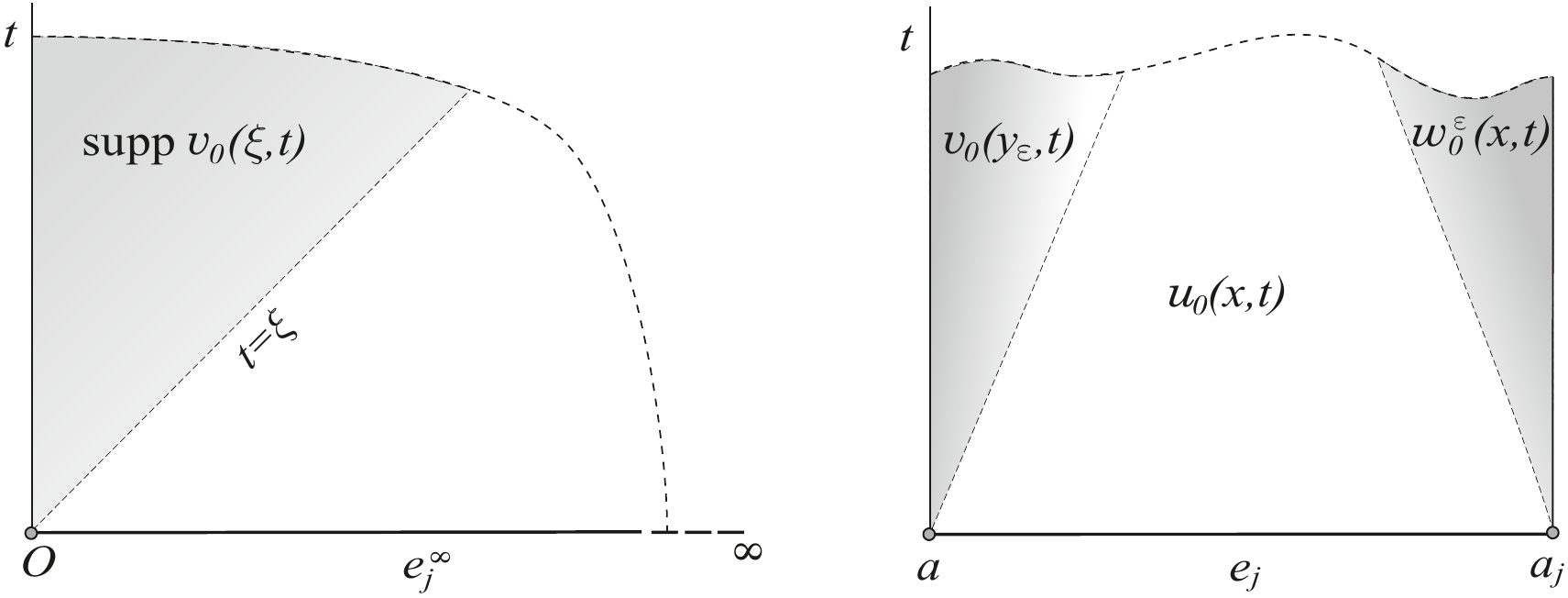}\\
  \caption{The boundary layers near the vertices $a$ and $a_j$ on the edge  $e_j$ of  subgraph $\cG_*$. The function $v_0$ is a solution of \eqref{ProblemV0} and $w_0^\eps$ is defined by \eqref{ProblemW0},\eqref{Weps0}.}\label{FigBoundaryLayers}
\end{figure}

We define $w_{0}^\eps\colon \cQ_*\to\Real$ as follows.
For any $e_j=(a,a_j)$ such that $e_j\in \cE_i$, $i=1,\dots,k$, the restriction of  $w_{0}^\eps$ to the set $e_j\times\cI$ coincides with the function
\begin{equation}\label{Weps0}
  w_{0,e_j}\left(\eps^{-m_i}(x-a_j),t\right).
\end{equation}
Hence, $w_{0}^\eps$ is a boundary layer function, which describes the singular behaviour of $\ue$ in a neighbourhood of boundary $\partial\cG_*$.
Note the boundary conditions at $z=0$ have been chosen in the  manner that $u_{0}+w_{0}^\eps$ satisfies \eqref{BoundCond}.

We will from now on assume that the time variable $t$ belongs to a finite interval $(0,T)$, and introduce the notation  $\cI_T=(0,T)$,
$\cQ_T=\cG\times\cI_T$,  $\cQ_0^T=\cG_0\times\cI_T$, and $\cQ_*^T=\cG_*\times\cI_T$.
Then if $\eps$ is small enough and $t\in (0,T)$, the functions $v_0(y_\eps,\,\cdot\,)$ and $w_0^\eps$ are actually boundary layers localized
near the vertex of $\cG_*$ (see the second plot in Fig.~\ref{FigBoundaryLayers}).
Hence, we set
\begin{equation}\label{LeadingTerms}
  \ue(x,t)\approx
  \begin{cases}\displaystyle
    U_0(x,t)& \text{if } (x,t)\in \cQ_0^T,\\
    u_0(x,t)+v_0(y_\eps,t)+w_0^\eps(x,t)& \text{if } (x,t)\in \cQ_*^T.
  \end{cases}
\end{equation}
We will refer to  \eqref{LeadingTerms} as the leading behavior or leading terms of the asymptotics.
At this point, this approximation produces the largest error of order $\eps^{m_1}$ in the Kirchhoff condition \eqref{Kirchhoff}.
We end the section by construction of the first correctors on $\cG_0$,  although we already solved the task declared in the title of  this section.
The problems for these correctors are slightly different from the problem for $U_0$.

\subsection{Correction of approximation error in the Kirchhoff condition}
To improve the accuracy of the approximation in the Kirchhoff condition  \eqref{Kirchhoff},
we finally add new terms to the approximation on $\cG_0$:
\begin{equation}\label{ApproxU0U1}
  \ue(x,t)\sim
  \begin{cases}\displaystyle
    U_0(x,t)+\sum_{i=1}^k \eps^{m_i}U_1^{(i)}(x,t)& \text{if } (x,t)\in \cQ_0^T,\\
   u_0(x,t)+v_0(y_\eps,t)+w_0^\eps(x,t)& \text{if } (x,t)\in \cQ_*^T.
  \end{cases}
\end{equation}
Our substituting \eqref{ApproxU0U1} into \eqref{Kirchhoff} yields
\begin{multline}\label{U1inKirchhoff}
   \sum\limits_{e\in\cE_0}\partial_x U_{0e}(a,t)+
    \sum_{i=1}^k \eps^{m_i}\sum\limits_{e\in\cE_0}\partial_x U_{1e}^{(i)}(a,t)\\
   +\sum_{i=1}^k \eps^{m_i}\sum\limits_{\gamma\in\cE_i^\infty} \partial_{\xi} v_{0\gamma}(O,\cdot)
   +\sum_{i=1}^k \eps^{2m_i}\sum\limits_{\gamma\in\cE_i} \partial_{x}u_{0e}(a,t)\sim 0.
\end{multline}
The function $w_0^\eps$ is absent in the last formula, because it vanishes in a neighbourhood of the vertex $a$.
The first sum in \eqref{U1inKirchhoff} is zero due to \eqref{ProblenU0G1last}. The next two double sums have to eliminate each other.   Therefore the function $U_1^{(i)}$ must be a solution to the problem
\begin{equation}\label{ProblemU1G0}
\begin{aligned}
  &\partial^2_t U-\partial^2_x U+qU= 0 \qquad \text{ in }  \cQ_0,\\
  &U=0,  \quad \partial_tU=0 \;\text{ on } \cG_0\times\{0\},\quad
  U=0  \; \text{ on } \partial\cG_0\times\cI,\\
  &U(\,\cdot\,,t) \text{  is continuous at } a \text{  for all } t\in\cI,\\
  &\sum\limits_{e\in\cE_0} \partial_x U_{e}(a,\cdot)=
  -\sum\limits_{\gamma\in\cE_i^\infty} \partial_{\xi} v_{0\gamma}(O,\cdot)
   \quad \text{ on } \cI,
\end{aligned}
\end{equation}
where $i\in\{1,\dots,k\}$. The problem is a partial case of the hyperbolic boundary value problem on $\cG_0$ with the nonhomogeneous Kirchhoff condition:
\begin{equation}\label{ProblemWithNHKirchhoff}
\begin{aligned}
  &\partial^2_t u-\partial^2_x u+qu= 0 \; \text{ in }  \cQ_0,\\
  &u=\phi,  \; \partial_tu=\psi \;\text{ on } \cG_0\times\{0\},\quad
  u=\mu  \; \text{ on } \partial\cG_0\times\cI,\\
  &u(\,\cdot\,,t) \text{  is continuous at } a \text{  for } t\in\cI,\quad
  \sum\limits_{e\in\cE_0} \partial_x u_{e}=\nu
  \; \text{ for } \{a\}\times\cI.
\end{aligned}
\end{equation}
It is reasonably easy to see that the  $C^1$ compatibility conditions for the problem
have the form:
 \begin{align}\label{CompatibilityCondNonhomoKirch}
 \begin{aligned}
   &\phi(a_i)=\mu(a_i,0), \quad \psi(a_i)=\partial_t\mu(a_i,0) \quad\text{ for } a_i\in\partial\cG_0,\\
   &\phi, \psi \text{ are continuous at }x=a, \quad \sum_{e\in\cE_0} \phi_e'(a)=\nu(0),
 \end{aligned}
\end{align}
provided $\nu$ is a continuous function on $\cI$.
And indeed these conditions hold for   \eqref{ProblemU1G0}, since $v_{0\gamma}(y,0)=0$ for all $\gamma\in \cE_i^\infty$.
Problem \eqref{ProblemWithNHKirchhoff} admits a unique solution $u\in \cU(\cQ_0)$ \cite{FlyudGolovaty2012}, see also the proof of Lemma~\ref{LemmaAprioryEst}.

By construction,  approximation \eqref{ApproxU0U1} satisfies conditions \eqref{InitCond}-\eqref{Continuity} exactly and \eqref{Kirchhoff} up to terms of order $\eps^{2m_1}$, since $m_1$ is the smallest number among the powers $m_i$, $i=1,\dots,k$. Besides, $v_0$ and $w_0^\eps$ cause small errors in the right hand side of  \eqref{Equation} which are localized in the boundary layer regions.

\section{Formal asymptotic expansions: general terms}\label{Sect4}
Taking into account our work in the previous section, we will look  for  a complete asymp\-totic expansion of $\ue$  in the form
\begin{align}\label{UeExpansionG0}
  &\ue(x,t)\sim
    U_0(x,t)+\sum_{s=1}^\infty\sum_{i=1}^k \eps^{s m_i}U_s^{(i)}(x,t)&& \text{if } (x,t)\in \cQ_0^T,\\\label{UeExpansionG1}
   &\ue(x,t)\sim \sum_{s=0}^\infty \{b^\eps(x)\}^{\frac{s}{2}}\big(u_s(x,t)+v_s(y_\eps,t)+w_s^\eps(x,t)\big)&& \text{if } (x,t)\in \cQ_*^T.
\end{align}
We also assume that all $v_s$ and $w_s^\eps$ are functions of the boundary layer type, i.e., for small $\eps$ the functions $v_s(y_\eps,t)$ are different from zero in a small neighbourhood of the central vertex and $w_s^\eps$ can possess non-zero values in a vicinity of the boundary $\partial\cG_*$ only. The validity of these assumptions will be considered further when we will construct the asymptotics. Recall that $\{b^\eps(x)\}^{\frac{s}{2}}=\eps^{s m_i}$ for $x\in \cG_i$.

 This may happen that some terms in \eqref{UeExpansionG0} as well as in \eqref{U1inKirchhoff} have the same order of smallness. This is because there are equal numbers among the integers $s m_i$. 
Let us introduce the sets $$\Lambda(p)=\{(n,i)\in\mathbb{N}\times\{1,\dots,k\}\colon n m_i=p \}$$
for each natural $p$. Of course, some of them are empty. All sets $\Lambda(p)$ are finite, therefore the terms of the same order in \eqref{UeExpansionG0} can be aggregated in the final form of approximation:
\begin{equation*}
  \sum_{s=1}^\infty\sum_{i=1}^k \eps^{s m_i}U_s^{(i)}=
  \sum_{p=1}^\infty\kern2pt\eps^{p}\kern-4pt\sum_{(s,i)\in\Lambda(p)}\kern-6ptU_s^{(i)}
\end{equation*}

Substituting \eqref{UeExpansionG0} and \eqref{UeExpansionG1} into {\eqref{Equation}--\eqref{Kirchhoff}} and collecting powers of $\eps$ give a sequence of problems for the terms of series. First, the functions $U_s^{(i)}$ are  solutions to the problems
\begin{equation}\label{ProblemUsG0}
\begin{aligned}
  &\partial^2_t U_s^{(i)}-\partial^2_x U_s^{(i)}+qU_s^{(i)}=0 \qquad \text{ in }  \cQ_0,\\
  &U_s^{(i)}=0,  \quad \partial_tU_s^{(i)}=0 \;\text{ on } \cG_0\times\{0\},\quad
  U_s^{(i)}=0  \; \text{ on } \partial\cG_0\times\cI,\\
  &U_s^{(i)}(\,\cdot\,,t) \text{  is continuous at } a \text{  for all } t\in\cI,\\
  & \sum\limits_{e\in\cE_0} \partial_x U_{s,e}^{(i)}=
 -\sum\limits_{e\in\cE_i} \partial_x u_{s-2,e}
 -\sum\limits_{\gamma\in\cE_i^\infty} \partial_{\xi} v_{s-1,\gamma}(O,\cdot)
 \; \text{ on } \{a\}\times\cI
\end{aligned}
\end{equation}
for $s\geq 2$ and $i\in\{1,\dots,k\}$. In our approximation the terms $U_s^{(i)}$ reduce  errors in  the  Kirchhoff condition \eqref{Kirchhoff}, but
without modification of continuity condition \eqref{Continuity}  as well as the initial and boundary value conditions \eqref{InitCond}, \eqref{BoundCond}.  Recall that each subgraph $\cG_i=\{\cV_i,\cE_i\}$ is associated with the non-compact graph  $\cG_i^\infty=\{\cV_i^\infty,\cE_i^\infty\}$, obtained by the  dilatation $\xi=\eps^{-m_i}(x-a)$ as $\eps\to0$. The point $O$ is the origin of  plane $\Real^2$ with the variables $\xi$.

The functions $u_s$ are terms of the regular asymptotics on subgraphs $\cG_1,\dots,\cG_k$.
Since the hyperbolic equation \eqref{Equation} becomes degenerate for $\eps=0$, we can find $u_s$ on each edge $e\in \cE_*$ by solving the Cauchy problems for ordinary differential equations with respect to time
\begin{equation}\label{ProblemUsG2}
\partial^2_t u_s+qu_s= \partial^2_x u_{s-2} \; \text{ in } e\times\cI,\qquad
 u_s=0,  \quad \partial_tu_s=0 \; \text{ on } e\times\{0\}
\end{equation}
for $s\geq1$, where $u_0$ is a solution of \eqref{ProblemU0G2} and $u_{-1}=0$. The task of $u_s$ is to exhaust the residual in the right-hand side of  \eqref{Equation} on more flexible graphs $\cG_i$, $i\geq 1$.
\begin{rem}
Equation \eqref{ProblemUsG2}  is homogeneous  for $s=1$. Then $u_1=0$ by uniqueness, and recursive calculations yield $u_3=0, u_5=0,\dots$.
Hence all $u_s=0$ with the odd index $s$ are zero functions.
\end{rem}

Next, the continuity condition \eqref{Continuity}  can be satisfied asymptotically by means of the boundary layer functions $v_s=v_s(y_\eps,t)$ which are localized about the central vertex.  Given  $i\in\{1,2,\dots,k\}$ and an edge $e\in \cE_i$, we consider the corresponding edge-ray $\gamma$ of $\cG^\infty_i$  and define the restriction of $v_s$ to $\gamma\times \cI$ as a solution of the problem
\begin{align}\label{ProblemVsEqn}
&\partial^2_t v_s -\partial^2_{\xi}v_s+ q_{\gamma}(a)v_s=
-\sum_{r=1}^s\tfrac{1}{r!}\,q^{(r)}_{\gamma}(a)\,\xi^r \,v_{s-r,\gamma}\quad\text{ in } \gamma\times \cI,\\
\label{ProblemVsIniConds}
&v_s=0,\quad \partial_t v_s=0 \; \text{ on } \gamma\times \{0\}, \\
\label{ProblemVsBndConds}
&v_s=\sum_{(r,l)\in\Lambda(sm_i)}\kern-8ptU_r^{(l)}(a,\cdot)-u_{s,e}(a,\cdot) \; \text{ on } \{O\}\times \cI
\end{align}
for all $s=1,2,\dots$.

Finally, we define the map $w_{s}^\eps\colon \cQ_*\to\Real$ that is a collection of boundary layers on each edge $e\in \cE_*$. These boundary layers are localized near the boundary $\partial\cG_*$ and reduce the approximation error in boundary condition \eqref{BoundCond}.
For any $i=1,\dots,k$ and $e\in \cE_i$,  the restriction of $w_{s}^\eps$ to $e$ is given by
\begin{equation*}
  w_{s,e}^\eps(x,t)=w_{s,e}\left(\eps^{-m_i}(x-a_j),t\right),
\end{equation*}
where $w_{s,e}$ solves the problem
\begin{equation}\label{ProblemWs}
\begin{aligned}
&\partial^2_t w_{s,e} -\partial^2_{z}w_{s,e}+ q_{e}(a_j)w_{s,e}=-\sum_{r=1}^s\tfrac{1}{r!} q^{(r)}_{e}(a_j)\,z^r \,w_{s-r,e}\; \text{ in } P,\\
&w_{s,e}(z,0)=0,\quad \partial_t w_{s,e}(z,0)=0 \quad\text{ for } z<0, \\
&w_{s,e}(0,t)=-u_{s,e}(a_j,t) \quad\text{ for } t>0
\end{aligned}
\end{equation}
for all $s\in \mathbb{N}$.
Here $P$ is the quarter-plane $\{(z,t)\colon z<0, \,t>0\}$ and $z=\eps^{-m_i}(x-a_j)$ is a new fast variable on the edge $e$.

\begin{lem}
  Suppose that the initial data of problem \eqref{Equation}--\eqref{Kirchhoff} satisfy conditions \eqref{AssumptionFQ} and \eqref{Fitting01}.
  Then the coefficients $U_s^{(i)}$, $u_s$, $v_s$ and $w_s^\eps$  of formal series \eqref{UeExpansionG0}, \eqref{UeExpansionG1} are unique and can be determined recursively up to arbitrary order $s$ in the class of  continuously differentiable functions which possess locally bounded derivatives of the second order. Furthermore, for $\eps$ small enough and $t\in (0,T)$ the functions $v_s(y_\eps,\,\cdot\,)$ and $w_s^\eps$ are boundary layer corrections which are localized about the central vertex   $a$ and the boundary $\partial\cG_*$ respectively.
\end{lem}
\begin{proof}
  The proof is by induction on $s$. Assume that we have already found the terms  $U_r^{(i)}$, $u_r$, $v_r$ and $w_r^\eps$ for $r<s$ with desired smoothness, where $v_r(y_\eps,\,\cdot\,)$ and $w_r^\eps$ are functions of boundary layer type as $\eps\to 0$. In order to show that the $s$-th step in the recursive process is solvable, we must analyze solvability of problems \eqref{ProblemUsG0}--\eqref{ProblemWs} in  appropriate spaces.

We start with problem \eqref{ProblemUsG2} which admits the explicit solution: $u_s(x,t)=0$ for odd $s$, and
\begin{align}\label{SolOfusQ}
  &u_s(x,t)=\frac{1}{\sqrt{q(x)}} \int_0^t \partial^2_x u_{s-2}(x,\tau)\sin\sqrt{q(x)}(t-\tau)\,d\tau, && \text{ if }  q(x)\neq 0,
 \\\label{SolOfusQ1}
 &u_s(x,t)=\int_0^t  (t-\tau)\partial^2_x u_{s-2}(x,\tau)\,d\tau,&& \text{ if }  q(x)=0
\end{align}
for even $s$ and $(x,t)\in \cQ_*$.
As shown in the proof of Lemma~\ref{LemmaU}, a solution $u_0$ of \eqref{ProblemU0G2} is a smooth function. Hence, all $u_s$ are also $\dot{C}^\infty(\cQ_*)$-functions due to the recursive procedure.

Hyperbolic problem \eqref{ProblemUsG0}  on graph $\cG_0$ contains the non-homogeneous Kirchhoff condition, where the  right hand side is known by induction hypothesis. The problem admits a unique solution $U_s^{(i)}\in \cU(\cQ_0)$, provided the $C^1$ compatibility conditions \eqref{CompatibilityCondNonhomoKirch} hold. In this case, the conditions take the form
\begin{equation*}
\sum\limits_{e\in\cE_i} \partial_x u_{s-2,e}(a,0)+
\sum\limits_{\gamma\in\cE_i^\infty} \partial_{\xi} v_{s-1,\gamma}(O,0)=0,
\end{equation*}
but the last equality is true, because  for $s>2$ the terms $u_{s-2}$ and $v_{s-1}$ satisfy the homogeneous initial conditions at $t=0$ by \eqref{ProblemUsG2} and \eqref{ProblemVsIniConds}. For $s=2$ we have
\begin{equation*}
  \sum\limits_{e\in\cE_i} \partial_x u_{0,e}(a,0)=\sum_{e\in\cE_i} \phi_e'(a)=0  \quad\text{ for } i\in\{0,\dots,k\}
\end{equation*}
by the fitting conditions \eqref{Fitting01} and $v_{1,\gamma}(y,0)=0$ on
$\gamma$ by \eqref{ProblemVsIniConds}.

As for problem \eqref{ProblemVsEqn}--\eqref{ProblemVsBndConds}, we first of all note that all right hand sides are already defined. The problem is the classic mixed problem for hyperbolic equation on the half-line \cite{SaramotoI1970, SaramotoII1970}. There exists a unique solution $v_{s,\gamma}\in \cU(\gamma\times\cI)$ for any edge $\gamma\in \cE_i^\infty$ and $i\in\{1,\dots,k\}$ if the compatibility conditions \eqref{CompatibilityCondForV} hold, i.e.,
\begin{equation*}
  \sum_{(r,l)\in\Lambda(sm_i)}\kern-8ptU_r^{(l)}(a,0)=u_{s,e}(a,0), \qquad \sum_{(r,l)\in\Lambda(sm_i)}\kern-8pt \partial_tU_r^{(l)}(a,0)=\partial_tu_{s,e}(a,0).
\end{equation*}
Both the equalities are true, since $U_r^{(l)}$ and $u_{s,e}$ satisfy the homogeneous initial conditions at $t=0$ as solutions of \eqref{ProblemUsG0} and \eqref{ProblemUsG2} respectively. Next, the right hand side of \eqref{ProblemVsEqn} is identically zero under the characteristic $s-t=0$, by induction. Taking into account the homogeneous initial conditions \eqref{ProblemVsIniConds} and integral representation \eqref{IntegralRepresentForSoln} we conclude that $v_{s,\gamma}$ is  equal to zero for $s>t>0$.

 The similar considerations can be also applied to problem \eqref{ProblemWs}  on the half-line, for which there exists a unique solution $w_{s,e}\in \cU(P)$.
\end{proof}

\section{Justification of the asymptotic expansions}\label{Sect5}

In this section we will prove that  formal series \eqref{UeExpansionG0}, \eqref{UeExpansionG1} actually solve the singularly perturbed problem {\eqref{Equation}--\eqref{Kirchhoff}} in the sense of asymptotic approximation.

\subsection{Estimation of remainder terms}
We introduce the partial sums of \eqref{UeExpansionG0}, \eqref{UeExpansionG1}
\begin{equation}\label{UPonG}
  \uep(x,t)=\begin{cases}
    U_0(x,t)+\sum\limits_{s=1}^p\sum\limits_{i=1}^k \eps^{s m_i}U_s^{(i)}(x,t)& \text{if } (x,t)\in \cQ_0^T,\\
   \sum\limits_{s=0}^{p} \{b^\eps(x)\}^{\frac{s}{2}}\big(u_s(x,t)+v_s(y_\eps,t)+w_s^\eps(x,t)\big)& \text{if } (x,t)\in \cQ_*^T
  \end{cases}
\end{equation}
with all coefficients constructed in Sections \ref{Sect3} and \ref{Sect4}.
Substituting $\uep$ into singularly perturbed problem {\eqref{Equation}--\eqref{Kirchhoff}}, we can estimate the remainder terms in the equation and all conditions. A somewhat lengthy, but not complicated, computation shows that  $\uep$ is a solution of the problem
\begin{align}
\label{EquationUN}
    &\partial^{2}_t\uep -\partial_x( b^\eps \partial_x \uep) +q\uep=f+h^{\eps,p}\quad \text{ in } \cQ_T,  \\
\label{InitCondUN}
    &\uep=\phi,  \quad \partial_t\uep=\psi \quad \text{ on } \cG\times\{0\},\\
\label{BoundCondUN}
    &\uep=\mu  \quad \text{ on } \partial\cG\times\cI_T,\\
\label{ContinuityUN}
    &\uep(\,\cdot\,,t) \text{ is continuous at } x=a \text{ for all } t\in\cI_T,\\
\label{KirchhoffUN}
    & \sum_{e\in\cE} b_{e}^\eps (a)\, \partial_x \uep_e(a,t)=\nu^{\eps,p}(t) \quad \text{  for all } t\in\cI_T.
\end{align}
By construction  the function $\uep$ precisely satisfies the initial and boundary value conditions \eqref{InitCondUN}, \eqref{BoundCondUN} as well as the continuity condition \eqref{KirchhoffUN}. The term $h^{\eps,p}$ in right-hand side of \eqref{EquationUN} is different from zero in a neighbourhood of the central vertex $a$ and the boundary $\partial\cG_*$ only. Moreover, there exists a constant $c_1(T)$ such that
\begin{equation}\label{RemaiderFest}
  |h^{\eps,p}(x,t)|\leq c_1(T)\eps^{(p+1)m_1}\quad\text{ for all } (x,t)\in \cQ_T.
\end{equation}
The remainder $\nu^{\eps,p}$ in the Kirchhoff condition \eqref{KirchhoffUN} is a continuous function and
\begin{equation}\label{RemaiderNUest}
  |\nu^{\eps,p}(t)|\leq c_2(T)\eps^{(p+1)m_1}\quad\text{ for  } t\in \cI_T.
\end{equation}
Recall that $m_1$ is the smallest number in the set $\{m_1,\dots,m_k\}$.
In order to prove that the smallness of  remainders $h^{\eps,p}$ and $\nu^{\eps,p}$ implies the asymptotic smallness of the difference between the exact solution $\ue$ of  {\eqref{Equation}--\eqref{Kirchhoff}} and the approximation $\uep$ as $\eps\to0$ we need some estimates for  solutions of hyperbolic problems on graphs.

\subsection{A priori estimate}
Let us consider the initial boundary value problem on graph $\cG$
\begin{align}\label{APequationG}
   & \partial^{2}_tu -\partial_x( b(x) \partial_x  u) +q(x) u=f(x,t)\quad \text{ in } \cQ_T,  \\
   \label{APinitialCondG}
   & u=\varphi,  \quad \partial_t u=\psi \quad \text{ on } \cG\times\{0\},\\
   \label{APboundaryCondG}
   &  u=0 \quad \text{ on } \partial\cG\times\cI_T,\\
   \label{APcontinuituG}
   &  u(\,\cdot\,,t) \text{ is continuous at } x=a \text{ for all } t\in\cI_T,\\
   \label{APKirchoffG}
   &  \sum_{e\in\cE} b_{e} (a)\, \partial_x  u_e(a,t)=\nu(t) \quad \text{  for all } t\in\cI_T.
\end{align}

Throughout the section, $W_2^l(\Omega)$, $l= 0,1,\dots$, stands for the Sobolev space of functions
defined on a set $\Omega$ which belong to $L^2(\Omega)$ together with their derivatives up to order $l$. In particular, we say that a function $f$ belongs to the  Sobolev space $W_2^l(\cG)$ on graph $\cG$, if its restrictions $f_e$  belong to $W_2^l(e)$ for all edges $e\in\cE$.

\begin{lem}[A priori estimate]\label{LemmaAprioryEst}
  Assume that  the input data of \eqref{APequationG}--\eqref{APKirchoffG} satisfy  conditions \eqref{AssumptionFQ}, \eqref{CompatibilityCondNonhomoKirch} and the coefficient $b$ is constant on each edge $e\in \cE$. If $u$ is a solution of  \eqref{APequationG}--\eqref{APKirchoffG}, then
  \begin{equation}\label{APrioriEstG}
  \|u\|_{W_2^2(\cQ_T)} \leq C(T) \,\left( \|\phi\|_{W_2^1(\cG)}+\|\psi\|_{L_2(\cG)}+\|\nu\|_{L_2(\cI_T)}+\|f\|_{L_2(\cQ_T)} \right)
\end{equation}
for some constant $C(T)$.
\end{lem}

\begin{proof} The main idea of proof is to decompose the problem on graph $\cG$ into $n$ problems on edges, for which such estimate is a well-known result.

Let $u$ be a solution of \eqref{APequationG}--\eqref{APKirchoffG} belonging to $\cU(\cQ_T)$. We will denote by $\sigma$ the restriction of $u$ to the set $\{a\}\times\cI_T$. Then for each edge $e\in\cE$ the restriction $u_e$ is a solution to the problem
\begin{align}\label{APequationE}
   & \partial^{2}_tu_e -b_e \partial^2_\eta  u_e +q_eu_e=f_e\quad \text{ in } (0,\ell_j)\times\cI_T,  \\
   \label{APinitialCondE}
   & u_e(\eta,0)=\varphi_e(\eta),  \quad \partial_t u_e(\eta,0)=\psi_e(\eta), \quad \eta\in (0,\ell_j),\\
   \label{APboundaryCondE}
   & u_e(0,t)=\sigma(t), \quad u_e(\ell_j,t)=0 \quad t \in \cI_T,
 \end{align}
where $e$ connects the vertices $a$ and $a_j$, and $\eta$ is a point of the interval $(0,\ell_j)$.
It should be stressed that  $\sigma$ is the same function for all edges $e\in\cE$, which is a consequence of continuity condition \eqref{APcontinuituG}.

At the same time, the function $\sigma$ is a solution of the Volterra integral equation of the second kind
\begin{equation}\label{IntegrEquat}
  \sigma(t)-\int_{0}^{t}\mathcal{K}(t,\tau)\sigma(\tau)\,d\tau=F(t)
\end{equation}
with the continuous kernel $\mathcal{K}$ and  right-hand $F$.  This equation was derived and studied by authors in \cite{FlyudGolovaty2012} and the explicit representation of $F$ via the input data of \eqref{APequationG}--\eqref{APKirchoffG} was obtained.
Next, for a solution of \eqref{APequationE}--\eqref{APboundaryCondE} we have the estimate \cite[IV.4]{LadyzhenskayaBVPs}
  \begin{equation}\label{APrioriEstEdge}
  \|u_e\|_{W_2^2(e\times\cI_T)} \leq C_1(T) \,\left( \|\phi\|_{W_2^1(e)}+\|\psi\|_{L_2(e)}+\|\sigma\|_{L_2(\cI_T)}+\|f\|_{L_2(e\times\cI_T)} \right),
\end{equation}
where $e\in\cE$. On the other hand, the equation \eqref{IntegrEquat} admits a unique solution $\sigma$ such that
 \begin{equation}\label{APrioriEstIntegral}
  \|\sigma\|_{L_2(\cI_T)} \leq C_2(T) \,\left( \|\phi\|_{W_2^1(e)}+\|\psi\|_{L_2(e)}+\|\nu\|_{L_2(\cI_T)}+\|f\|_{L_2(e\times\cI_T)} \right),
\end{equation}
since the right-hand side $F$ depends on the input data of \eqref{APequationG}--\eqref{APKirchoffG}.
Inequalities \eqref{APrioriEstEdge} and \eqref{APrioriEstIntegral} combined give the estimate \eqref{APrioriEstG} after summation over $e\in\cE$.
\end{proof}


\subsection{Justification of asymptotics}

We now have the desired result.

\begin{thm}
  Given $T>0$ suppose that  $\ue$ is a weak solution of problem  {\eqref{Equation}--\eqref{Kirchhoff}} in $\cQ_T$. Then $\ue$ admits the asymptotic expansion of the form
  \begin{equation*}
     \ue(x,t)\sim
     \begin{cases}
       \;U_0(x,t)+\sum\limits_{s=1}^\infty\sum\limits_{i=1}^k \eps^{s m_i}U_s^{(i)}(x,t) & \text{if } (x,t)\in \cQ_0^T,\\
       \;\sum\limits_{s=0}^\infty \{b^\eps(x)\}^{\frac{s}{2}}\big(u_s(x,t)+v_s(y_\eps,t)+w_s^\eps(x,t)\big)& \text{if } (x,t)\in \cQ_*^T
     \end{cases}
  \end{equation*}
 with all coefficients constructed in Sections \ref{Sect3} and \ref{Sect4}, namely for any $p=0,1,\dots$  the solution $\ue$ and the approximation $\uep$ given by \eqref{UPonG} satisfy the inequality
 \begin{equation}\label{AsymptoticEst}
   \|\ue-\uep\|_{W_2^2(\cQ_T)}\leq C_p(T)\,\eps^{(p+\frac{1}{2})m_1}
 \end{equation}
with constant $C_p(T)$, being independent of $\eps$.
\end{thm}


\begin{proof}
Our proof starts with the observation that the difference $\uep-\ue$ solves the problem \eqref{EquationUN}--\eqref{KirchhoffUN} with $\phi$, $\psi$, $\mu$ and $f$ replaced by zero functions. Therefore this difference can be estimated by the corresponding norms of remainder terms $h^{\eps,p}$ and $\nu^{\eps,p}$ using  estimate \eqref{APrioriEstG}. But one must be careful with this estimate, because the constant $C(T)$ is inversely proportional to the ellipticity bound of \eqref{APequationG}. In the case of problem \eqref{EquationUN}--\eqref{KirchhoffUN} we have $\min_{x\in\cG}b^\eps(x)=\eps^{2m_k}$ for small $\eps$, where $m_k=\max_i\{m_i\}$. Hence,
  \begin{equation*}
    \|\ue-\uep\|_{W_2^2(\cQ_T)}\leq c_1\eps^{-2m_k} \left(\|\nu^{\eps,p}\|_{L_2(\cI_T)}+\|h^{\eps,p}\|_{L_2(\cQ_T)} \right)\leq c_2\eps^{(p+1)m_1-2m_k},
  \end{equation*}
by \eqref{RemaiderFest} and \eqref{RemaiderNUest}. In order to improve the estimate, we consider the last inequality for bigger number $p+r$. Then
 \begin{equation*}
    \|\ue-u^{\eps,p+r}\|_{W_2^2(\cQ_T)}\leq  c_2\eps^{(p+1)m_1+(rm_1-2m_k)}\leq  c_3\eps^{(p+1)m_1},
  \end{equation*}
provided $rm_1\geq 2m_k$. From this we readily deduce  that
 \begin{equation*}
    \|\ue-u^{\eps,p}\|_{W_2^2(\cQ_T)}-\|u^{\eps,p+r}-\uep\|_{W_2^2(\cQ_T)}\leq  c_3\eps^{(p+1)m_1},
  \end{equation*}
and hence that
 \begin{equation*}
    \|\ue-u^{\eps,p}\|_{W_2^2(\cQ_T)}\leq  c_3\eps^{(p+1)m_1}+\|u^{\eps,p+r}-\uep\|_{W_2^2(\cQ_T)}.
  \end{equation*}
It is worth remarking at this stage that the $W_2^2$ norms of boundary layers $v_s(y_\eps,\,\cdot\,)$ and $w_s^\eps$ are  infinite large as $\eps\to0$. Moreover, the contributions in the norm from the second derivatives on $x$ are the largest. Therefore
\begin{equation*}
    \|u^{\eps,p+r}-\uep\|_{W_2^2(\cQ_T)}\leq c_4 \left(\|(b^\eps)^{\frac{p+1}{2}}\partial^2_xv_{p+1}(y_\eps,\,\cdot\,)\|_{L^2(\cQ_*^T)}
    +\|(b^\eps)^{\frac{p+1}{2}}\partial^2_xw_{p+1}^\eps\|_{L^2(\cQ_*^T)}\right).
\end{equation*}
We will derive an estimate for the first norm in the right hand side of the last inequality. The second one can be bounded similarly.
We have
\begin{multline*}
  \|(b^\eps)^{\frac{p+1}{2}}\partial^2_xv_{p+1}(y_\eps,\,\cdot\,)\|_{L^2(\cQ_*^T)}^2 =
  \sum_{i=1}^k \eps^{2m_i(p+1)}\sum\limits_{\gamma\in\cE_i} \|\partial^2_xv_{p+1,\gamma}(y_\eps,\,\cdot\,)\|_{L^2(\gamma\times\cI_T)}^2
\\  \leq c_5 \sum_{i=1}^k \eps^{2m_i(p+1)}\eps^{-m_i}\leq c_6 \eps^{(2p+1)m_1},
\end{multline*}
since
\begin{multline*}
  \|\partial^2_xv_{s,\gamma}(y_\eps,\,\cdot\,)\|_{L^2(\gamma\times\cI_T)}^2=\int_0^T\int_\gamma|\partial^2_xv_{s,\gamma}(\eps^{-m_i}(x-a_j) ,t)|^2\,d\gamma\,dt \\
  =\eps^{-2m_i}\int_0^T\int_0^{\eps^{m_i}}|\partial^2_\xi v_{s,\gamma}(\eps^{-m_i}\pi_\gamma(\alpha) ,t)|^2\,d\alpha\,dt\\
  =\eps^{-m_i}\int_0^T\int_0^1|\partial^2_\xi v_{s,\gamma}(\xi ,t)|^2\,d\xi\,dt\leq c_7\eps^{-m_i}
\end{multline*}
for any $\gamma\in \cE_i$, where $\pi_\gamma\colon [0,\ell_j]\to \gamma$ is the natural parametrization of the  edge $\gamma$.
Here we also used the main property of boundary layer $v_s(y_\eps,\,\cdot\,)$  to be different from zero in the $\sqrt{b^\eps}\, t$-neigh\-bour\-hood of vertex $a$ only.
Hence
\begin{equation}\label{Up+r-UpEst}
   \|u^{\eps,p+r}-\uep\|_{W_2^2(\cQ_T)}\leq c_8\, \eps^{(p+\frac{1}{2})m_1}.
\end{equation}
Combining the previous inequalities now yields \eqref{AsymptoticEst}.
\end{proof}

\paragraph{\bf Acknowledgment}
The authors would like to thank Ruslan Andrusiak for helpful discussions. We are also greatly indebted to  Referee
for carefully reading the paper and suggesting some improvements.

\bibliographystyle{amsplain}

\end{document}